\documentclass[12pt,reqno]{amsart}
\usepackage[margin=1in]{geometry}
\usepackage[normalem]{ulem}
\usepackage{amsmath,amsthm,amssymb,color,latexsym}
\usepackage{graphicx}
\usepackage[shortlabels]{enumitem}
\usepackage[export]{adjustbox}
\usepackage{tabularray}
\usepackage{multirow}
\usepackage{thmtools} 
\usepackage[hypertexnames=false]{hyperref} 
\usepackage{arydshln}
\usepackage{cleveref}
\crefname{appendix}{Appendix}{Appendices}
\Crefname{appendix}{Appendix}{Appendices}

\theoremstyle{plain}
\newtheorem{thm}{Theorem}[section]
\newtheorem{prop}[thm]{Proposition}
\newtheorem{lem}[thm]{Lemma}
\newtheorem{crl}[thm]{Corollary}

\theoremstyle{definition}
\newtheorem{defn}[thm]{Definition}
\newtheorem{ex}[thm]{Example}
\newtheorem{rmk}[thm]{Remark}

\numberwithin{equation}{section}

\newcommand\C{\mathbb{C}}
\newcommand\R{\mathbb{R}}

\def\e{\varepsilon}\def\g{\gamma}\def\G{\Gamma}\def\l{\lambda}
\def\b{\beta}
\def\k{\kappa}
\def\T{\Theta}\def\t{\theta}
\def\L{\Lambda}\def\tL{\widetilde{\Lambda}}

\def\cL{\mathcal{L}}

\def\cH{\mathcal{H}}\def\cD{\mathcal{D}}
\def\cK{\mathcal{K}}

\DeclareMathOperator{\RE}{Re}

\DeclareMathOperator{\Ran}{Ran}

\renewcommand{\geq}{\geqslant}
\begin{document}

\title[Eigenvalue escape rates on graphs with a shrinking core]{Two escape rates for negative eigenvalues of quantum graphs with a shrinking core}

\author{Gregory Berkolaiko}
\address{Department of Mathematics, Texas A\&M University, College
  Station, TX 77843-3368, USA}
\author{Denis Borisov}
\address{Institute of Mathematics, Ufa Federal Research Centre, Russian Academy
  of Sciences, Ufa 450008, Russia; \newline
Bashkir State Pedagogical University named after M.~Akhmulla,
  Ufa 450000, Russia; \newline
Peoples' Friendship University of Russia (RUDN University),
  Moscow 117198, Russia}
\author{Marshall King}
\address{Department of Mathematics, Texas State University,
  San Marcos, TX 78666-4684, USA}
\author{Julien Royer}
\address{Institut de Math\'ematiques de Toulouse, Universit\'e de Toulouse,
  F-31062 Toulouse Cedex~9, France}

\begin{abstract}
  We study the negative spectrum of the Laplacian on a metric graph
  with general vertex matching conditions and with two length scales:
  a compact core whose edges have length of order a small parameter
  $\e$, together with finitely many edges of infinite length. As
  $\e\to0$, some negative eigenvalues may escape to $-\infty$, and we
  describe precisely how. There are exactly two rates of escape,
  $\e^{-1}$ and the fractional rate $\e^{-2/3}$. We determine the
  number of eigenvalues of each rate, together with their leading
  coefficients, explicitly from the vertex conditions. The analysis
  rests on the Dirichlet-to-Neumann map of the graph and on an
  implicit Rellich-type theorem, that identifies the power-law rates
  of the solution branches of a nonlinear 2-parameter matrix pencil
  with the leading orders of a one-parameter Hermitian family.
\end{abstract}


\dedicatory{Dedicated to Pavel Exner on the occasion of his 80th birthday.}

\maketitle

\section{Introduction}\label{sec:intro}
Differential operators on metric graphs model thin branching structures across
physics and engineering, from conducting wires \cite{KosSch_jpa99,SmiSol_conm06,Car_nhm11}
and optical waveguides \cite{Kuc_incol01,Ong_diss} to diffusion in thin tubes
\cite{Alb_anp12}, beam frames \cite{WilWit_ijms70,LR13,Mei19}, and networks in
biology \cite{Nic_incol85,SarCarAnd_jmb14}. In each case the object of interest
is the spectrum of a Sturm--Liouville or Laplace operator acting on the edges,
augmented by matching conditions at the vertices \cite{BerKuc_graphs,Mugnolo_book}.
The topology of the network makes the spectral theory rich: eigenmodes can
localize exactly on a subgraph \cite{SchKot_prl03,CdV_ahp15}, playing havoc with
inverse and control problems \cite{Kur_jmp13}, though for the Laplacian with
standard conditions such exact localization is destroyed by generic
perturbations of the edge metric \cite{BerLiu_jmaa17}.

This work is motivated by networks whose edge lengths span two scales:
some edges are much shorter than others, their length of order a small
parameter $\e$. Such models arise in the study of metamaterials, where
small-scale inclusions substantially alter the bulk properties
\cite{CheExnTur_anp10,DoKucOng_ems17,LawTanChr_sr22}, and in real
networks whose short links coexist with long-range connections. The
natural question is the behavior of the spectrum as $\e\to0$. While
the spectrum of a compact quantum graph depends analytically on its
edge lengths \cite{BerKuc_incol12,KucZha_jmp19}, that result
explicitly excludes edges shrinking to a point, and it is precisely
this degenerate limit that we study.

Graphs with edges of length scale $\e$ and general self-adjoint vertex
matching conditions have received sustained attention recently. Two
questions on the behavior of the eigenvalue spectrum have been
addressed: the convergence on any compact set, and the asymptotics of
eigenvalues escaping to $\pm\infty$ (and thus leaving any compact set).
For standard vertex matching conditions, convergence was established
by Band and L\'evy \cite{BanLev_ahp17}.  For general matching
conditions, Cacciapuoti \cite{Cac_s19} and Berkolaiko, Latushkin, and
Sukhtaiev \cite{BerLatSuk_am19}, later substantially generalized by
Borisov \cite{Bor_am22,Bor_math21}, established $\e\to0$ convergence of spectrum
on any compact set under a \emph{non-resonance condition} preventing
eigenfunctions from localizing on the shrinking part of the graph.
When the condition is dropped, the limit can fail.  This failure was
further investigated for graphs with \emph{all} edges shrinking to
zero: Berkolaiko and Colin de Verdi\`ere \cite{BerCdV_jmaa24}
classified all possible eigenvalue asymptotics, finding some
``exotic'' eigenvalues whose finite limit is an obstacle to the above
convergence and also finding that the eigenvalues escaping to
$-\infty$ do so only at the rate $\e^{-1}$.  Notably, this rate
saturates the a priori bound $\l\gtrsim-\e^{-1}$ of Kuchment
\cite{Kuc_wrm04,KosSch_incol06,BolEnd_ahp09}.

In this work, we continue to study the eigenvalues escaping to
$-\infty$ by taking up the case that is genuinely two-scale: a compact
core of edges of length scale $\e$ together with edges of infinite
length, see \Cref{fig:general}.  In a similar setting and for a particularly simple family of
graphs (one short edge and one long), it was found in
\cite{BerBorKin_aamp23} that the eigenvalues can also escape to
$-\infty$ at the \emph{fractional} rate $\e^{-2/3}$.  The fractional
power is striking, as nothing in the problem --- the Laplacian, the
vertex conditions, the edge lengths --- introduces one; it is
reminiscent of Simon's analysis of eigenvalues absorbed into the
continuous spectrum \cite{Sim_jfa77}, where such fractional rates are
produced by a fractional power of the Laplacian, whereas here the same
phenomenon arises for the plain Laplacian itself.
A natural question is \emph{what other fractional rates are
possible in this setting}?

Our main result, \Cref{thm:main}, shows that \emph{these two rates are the
only ones}, independently of the graph's topology or the vertex
conditions. More precisely, we establish, in terms of the vertex data alone:
\begin{itemize}
\item \emph{the rates:} every negative eigenvalue that escapes to $-\infty$ does
  so at one of exactly two rates, $\e^{-1}$ or $\e^{-2/3}$;
\item \emph{the count} of the eigenvalues of each rate, as the inertia
  of an explicit Hermitian matrix built from the vertex conditions;
\item \emph{the leading-order coefficients} of each escaping
  eigenvalue, as an eigenvalue of an explicit Hermitian matrix built
  from the vertex conditions.
\end{itemize}

There are two principal tools behind these results. The first is the
Dirichlet-to-Neumann (DtN) map (also known as the $M$-function) of
spectral theory, which recasts the eigenvalue problem in determinantal
form, as a real-analytic equation in two variables: the small
parameter and the eigenvalue. The second is a Hermitian matrix
analogue of the Weierstrass preparation theorem of algebraic geometry,
which we conjectured in the course of our work on the problem and
which was recently established by Dencker \cite{Den26}.  Dencker's
theorem yields an implicit Rellich-type theorem
(\Cref{thm:perturbedRellich}), which may be of independent interest:
it identifies the power-law rates of the solutions of a determinantal
real-analytic equation (in two variables) with the power-law rates of
the eigenvalues of its one-variable Hermitian section.  As a
consequence, we show the eigenvalue branches of the quantum graph to
be real-analytic in a suitable variable.

The paper is organized as follows. \Cref{sec:setting} fixes the graph
and the operator notation, and the vertex-data blocks that carry the
counts; \Cref{sec:main} states the main theorem and \Cref{sec:example}
works out a small example.  \Cref{sec:outline} outlines the proof.
The principal tools are introduced in \Cref{sec:DtN} (the
Dirichlet-to-Neumann reduction) and \Cref{sec:rellich} (implicit
Rellich theorem); \Cref{sec:proof} contains the proof itself.

\subsection*{Acknowledgements}
We thank Nils Dencker, Mark Goresky, Dmitry Kerner and Mikhail
Zaidenberg for their discussions of the conjectured (at that point of
time) factorization \eqref{eq:PRfact}; this conjecture was later
proved by Nils Dencker \cite{Den26} and now enables our implicit
Rellich-type theorem, \Cref{thm:perturbedRellich}.  We are grateful to
Dean Baskin, Jeremy Marzuola, and Carsten Trunk for discussions of the
possible reasons for the fractional escape rates, and possible ways to
prove it. We are especially grateful to Peter Kuchment for his
encouragement and his interest in this work --- and for his disbelief
in our conjecture that only two rates can occur, which was a great
stimulus to seek the complete picture. The collaboration of J.~R.\ and
G.~B.\ on this paper was made possible by the support of the Labex
CIMI, Toulouse, France (grant ANR-11-LABX-0040-CIMI).

\section{Setting and the main result}\label{sec:setting}

\subsection{The graph and the operator}
Let $\e>0$ be small. The graph $G_\e$ consists of a compact core with $m_0$
\emph{short edges} $e_1,\dots,e_{m_0}$, together with $m_1$ semi-infinite
\emph{leads} attached at some of its vertices, for a total of $m_0+m_1$ edges.
The short edges carry the small scale: the length of the
$j$-th one is
\begin{equation}\label{eq:edgelength}
    \ell_j=\e\,r_j,\qquad j=1,\dots,m_0,
\end{equation}
where the \emph{rescaled length} $r_j>0$ is a constant and $\e$ is the small
parameter, while the leads have infinite length. A representative graph is
shown in \Cref{fig:general}.
\begin{figure}[h]
    \centering
    \includegraphics[scale=0.8]{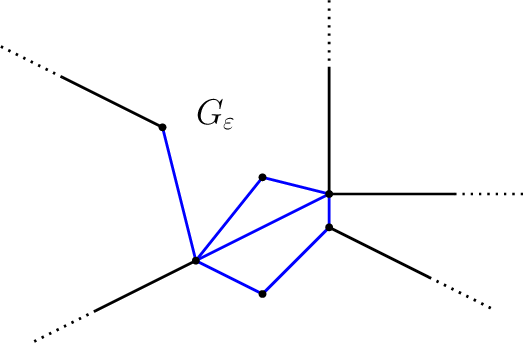}
    \caption{A graph $G_\e$ with $m_0=7$ short edges and $m_1=5$ infinite leads.}
    \label{fig:general}
  \end{figure}

Let $\cH_\e=-\Delta$ be the Laplacian on $G_\e$, acting edgewise on
\begin{equation}\label{eq:domain}
    \cD(\cH_\e)=H^2(G_\e)=\bigoplus_{j=1}^{m_0+m_1}H^2(e_j),
\end{equation}
and made self-adjoint by vertex conditions that are \emph{independent of $\e$}
(specified in \Cref{subsec:vertexcond}). The essential spectrum of $\cH_\e$ is
$[0,\infty)$; its negative spectrum consists of isolated eigenvalues of finite
multiplicity, and these are the object of this paper. An a priori bound of
Kuchment \cite{Kuc_wrm04} (see also \cite{KosSch_incol06,BolEnd_ahp09}) places
the bottom of the spectrum at $\gtrsim-1/\ell_{\min}$, where $\ell_{\min}$ is
the shortest edge length; since $\ell_{\min}=\e\min_j r_j$, every eigenvalue
obeys
\begin{equation}\label{eq:apriori}
    \l\gtrsim-\e^{-1},\qquad\text{equivalently}\qquad
    \k:=\sqrt{-\l}\lesssim\e^{-1/2}\quad(\l<0).
\end{equation}

\subsection{Vertex conditions}\label{subsec:vertexcond}
The vertex conditions couple the boundary values of $f\in H^2(G_\e)$ and of its
inward derivative through the Dirichlet and Neumann traces. Listing the two ends
of each short edge and the single end of each lead,
\begin{equation}\label{eq:DNtrace}
    F=\begin{pmatrix}
        f_1(0)\\ f_1(\ell_1)\\ \vdots\\ f_{m_0}(0)\\ f_{m_0}(\ell_{m_0})
    \end{pmatrix}\oplus\begin{pmatrix}
        f_{m_0+1}(0)\\ \vdots\\ f_{m_0+m_1}(0)
    \end{pmatrix},
    \qquad
    F'=\begin{pmatrix}
        f_1'(0)\\ -f_1'(\ell_1)\\ \vdots\\ f_{m_0}'(0)\\ -f_{m_0}'(\ell_{m_0})
    \end{pmatrix}\oplus\begin{pmatrix}
        f_{m_0+1}'(0)\\ \vdots\\ f_{m_0+m_1}'(0)
    \end{pmatrix},
\end{equation}
so that $F,F'\in\C^m$ with
\begin{equation}\label{eq:mdef}
    m=2m_0+m_1 .
\end{equation}
The signs in $F'$ make every derivative point into its edge. Self-adjoint
vertex conditions can then be written \cite{Kuc_wrm04} as
\begin{equation}\label{eq:vertexcond}
    (I_m-Q)F=0,\qquad QF'=\T\,QF ,
\end{equation}
where $Q$ is the orthogonal projector onto the Neumann--Robin part of the
conditions, $I_m-Q$ projects onto the Dirichlet part, and $\T:\Ran Q\to\Ran Q$
is self-adjoint. All of the vertex data is thus carried by the pair $(Q,\T)$.

The representation \eqref{eq:vertexcond} does not refer to the topology of
$G_\e$: it couples edge ends only through $Q$ and $\T$. Collapsing all vertices
into a single vertex (turning each short edge into a loop and each lead into a
ray, as in \Cref{fig:flower}) realizes any such conditions on the resulting
\emph{flower graph} without loss of generality. We therefore take $G_\e$ to be a
flower graph from now on.
\begin{figure}[h]
    \centering
    \includegraphics{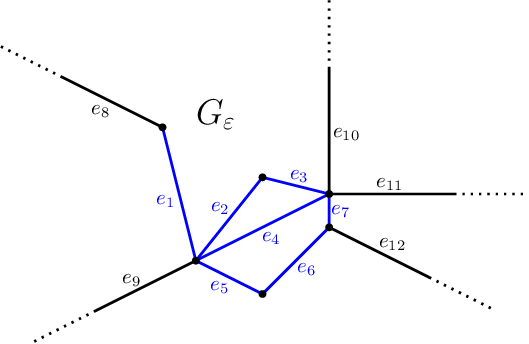}\qquad
    \includegraphics{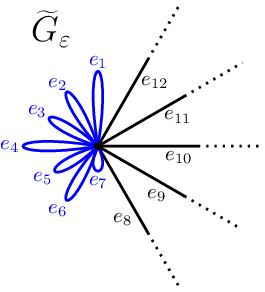}
    \caption{The graph $G_\e$ of \Cref{fig:general} and its flower graph
    $\widetilde G_\e$.}
    \label{fig:flower}
\end{figure}

\subsection{Decomposition of the trace space}\label{subsec:decomp}
The counts in the main theorem are read off a decomposition of
$\Ran Q$ adapted to how each trace direction meets the short scale. On
the two ends of the $j$-th loop, the trace space $\C^m$ carries the
orthonormal \emph{symmetric} and \emph{antisymmetric} directions
\begin{equation}\label{eq:loopeigvec}
    u_1=\tfrac1{\sqrt2}\begin{pmatrix}1\\1\end{pmatrix},
    \qquad
    u_2=\tfrac1{\sqrt2}\begin{pmatrix}1\\-1\end{pmatrix}.
\end{equation}
These organize $\C^m$ into three mutually orthogonal families: the symmetric
loop directions, the leads, and the antisymmetric loop directions,
\begin{align}\label{eq:cLdef}
  \cL_1&=\bigg(\bigoplus_{j=1}^{m_0}\operatorname{span}(u_1)\bigg)\oplus 0_{m_1}
    =\big\{(a_1,a_1,\dots,a_{m_0},a_{m_0},\,0,\dots,0)^{T}:a_j\in\C\big\},\\
  \cL_2&=0_{2m_0}\oplus\C^{m_1}
    =\big\{(0,\dots,0,\,b_1,\dots,b_{m_1})^{T}:b_k\in\C\big\},\nonumber\\
  \cL_3&=\bigg(\bigoplus_{j=1}^{m_0}\operatorname{span}(u_2)\bigg)\oplus 0_{m_1}
    =\big\{(a_1,-a_1,\dots,a_{m_0},-a_{m_0},\,0,\dots,0)^{T}:a_j\in\C\big\},\nonumber
\end{align}
with $\cL_1\oplus\cL_2\oplus\cL_3=\C^m$. In \Cref{sec:DtN} these turn out to be
exactly the directions on which the rescaled Dirichlet-to-Neumann map has orders
$s^2$, $s$, and $1$ as $s\to0$; see \eqref{eq:Ldecomp}.

\begin{lem}\label{lem:decomp}
  Let $Q$ be the orthogonal projector of the vertex conditions. Then
  \begin{equation}\label{eq:RanQdecomp}
    \Ran Q=X_1\oplus X_2\oplus X_3
  \end{equation}
  is an orthogonal decomposition, where
  \begin{equation}\label{eq:Xdef}
    X_1=\Ran Q\cap \cL_1,\qquad
    X_2=\Ran Q\cap (\cL_1\oplus\cL_2)\cap X_1^\perp,\qquad
    X_3=\Ran Q\cap ( X_1\oplus X_2 )^\perp.
  \end{equation}
\end{lem}

The proof is deferred to \Cref{subsec:reduction}, where the decomposition
organizes the block reduction of the determinant equation; the three pieces sort
$\Ran Q$ by how much of the small subspaces each direction sees, with
$X_1\subseteq\cL_1$ orthogonal to $\cL_2\oplus\cL_3$ and
$X_2\subseteq\cL_1\oplus\cL_2$ orthogonal to $\cL_3$.

We now describe the data necessary to describe the counts of the
eigenvalues escaping with each of the two rates and their leading
order coefficients.  Fix orthonormal frames $V_1,V_2,V_3$ for
$X_1,X_2,X_3$ and assemble
$V=\begin{pmatrix}V_1&V_2&V_3\end{pmatrix}$, an orthonormal frame for
$\Ran Q$.  Alongside $\T$ we will need the diagonal positive form
carrying the rescaled loop lengths on $\cL_1$,
\begin{equation}\label{eq:Rdef}
    R:=\bigg(\bigoplus_{j=1}^{m_0}\frac{r_j}{2}\,u_1u_1^*\bigg)\oplus 0_{m_1},
\end{equation}
which reappears as the leading coefficient $J_1(0)$ of the rescaled DtN map
(see \eqref{eq:J0}). The vertex data then enters the counts through the blocks
\begin{equation}\label{eq:blocks}
    \T_{1}=V_1^*\T V_1,\qquad
    \T_{21}=V_2^*\T V_1,\qquad
    L_1=V_1^*R\,V_1 ,
\end{equation}
where $L_1>0$ because $R$ is positive definite on $\cL_1\supseteq X_1$. Finally, with $W_2$ an orthonormal frame for $\ker\T_1$, set
\begin{equation}\label{eq:kerblocks}
    L_2:=W_2^*L_1W_2>0,\qquad
    \Xi:=W_2^*\T_{21}^*\T_{21}W_2=(\T_{21}W_2)^*(\T_{21}W_2)\ge0 ,
\end{equation}
the compressions of $L_1$ and $\T_{21}^*\T_{21}$ to $\ker\T_1$.

\subsection{Main result}\label{sec:main}
The negative spectrum of $\cH_\e$ is bounded below by a constant multiple of
$-\e^{-1}$ (see \eqref{eq:apriori}); our main result describes precisely the
negative eigenvalues that nonetheless escape to $-\infty$ as $\e\to0^+$. There
are exactly two rates of escape, and both the number of eigenvalues of each rate
and their leading coefficients are determined by the vertex data $(Q,\T)$.
Here $n_\pm(\cdot)$ denotes the
number of positive, respectively negative, eigenvalues.

\begin{thm}[Escaping negative eigenvalues]\label{thm:main}
    There exist non-negative integers $n_1,n_2$ such that, for all sufficiently
    small $\e>0$, the negative eigenvalues of $\cH_\e$ that do not remain
    bounded as $\e\to0^+$ split into two families, and no other rate of escape
    occurs:
    \begin{enumerate}[(i)]
        \item \emph{(Rate $\e^{-1}$.)} There are
        \begin{equation}\label{eq:n1}
            n_1=n_-(\T_1)
        \end{equation}
        eigenvalues
        \begin{equation}\label{eq:typeS}
            \l_i(\e)=-a_i\,\e^{-1}\big(1+O(\e^{1/2})\big),\qquad a_i>0,
        \end{equation}
        whose coefficients $a_i$, listed with multiplicity, are the positive
        eigenvalues of the Hermitian matrix $-L_1^{-1/2}\,\T_1\,L_1^{-1/2}$;
        \item \emph{(Rate $\e^{-2/3}$.)} There are
        \begin{equation}\label{eq:n2}
            n_2=\dim\ker\T_1-\dim(\ker\T_1\cap\ker\T_{21})=n_+(\Xi)
        \end{equation}
        eigenvalues
        \begin{equation}\label{eq:typeC}
            \l_i(\e)=-a_i\,\e^{-2/3}\big(1+O(\e^{1/3})\big),\qquad a_i>0,
        \end{equation}
        whose coefficients are the numbers $a_i=c^{2/3}$, with $c$ ranging, with
        multiplicity, over the positive eigenvalues of
        $L_2^{-1/2}\,\Xi\,L_2^{-1/2}$.
    \end{enumerate}
    Moreover each escaping branch $\l_i(\e)$ is, as $\e\to0^+$, a meromorphic
    function of the fractional variable $\e^{1/2}$ (first family), respectively
    $\e^{1/3}$ (second family), with a double pole at $\e=0$; the leading and
    first-correction orders are those shown in \eqref{eq:typeS},
    \eqref{eq:typeC}.
\end{thm}

The fractional rate $\e^{-2/3}$ is the phenomenon that distinguishes the present
two-scale setting from the case of a graph shrinking to a point, where only the
rate $\e^{-1}$ occurs \cite{BerCdV_jmaa24}; see \Cref{sec:example} for the smallest
graph exhibiting it.

\section{Example Illustrating the Main Results}\label{sec:example}
We illustrate \Cref{thm:main} on the smallest flower graph that still exhibits
both rates: one short loop and one lead.

\begin{ex}\label{ex:simplest}
    Let $m_0=m_1=1$ and $Q=I_3$, so that $\Ran Q=\C^3$ and there is no Dirichlet
    part; the graph is a single shrinking loop joined to one infinite lead
    (\Cref{fig:twoedge}). We write the self-adjoint vertex matrix as
    \begin{equation}\label{eq:exTheta}
        \T=\begin{pmatrix}
            \t_{11}&\t_{12}&\t_{13}\\
            \overline{\t_{12}}&\t_{22}&\t_{23}\\
            \overline{\t_{13}}&\overline{\t_{23}}&\t_{33}
        \end{pmatrix},
    \end{equation}
    the first two coordinates being the two ends of the loop and the third the
    end of the lead.
    \begin{figure}[h]
        \centering
        \includegraphics{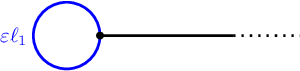}
        \caption{The graph of \Cref{ex:simplest}: one short loop and one lead,
        with only Neumann--Robin conditions at the central vertex.}
        \label{fig:twoedge}
    \end{figure}

    The symmetric loop direction, the lead, and the antisymmetric loop
    direction give the frames
    \begin{equation}
        V_1=\tfrac1{\sqrt2}\begin{pmatrix}1\\1\\0\end{pmatrix},\qquad
        V_2=\begin{pmatrix}0\\0\\1\end{pmatrix},\qquad
        V_3=\tfrac1{\sqrt2}\begin{pmatrix}1\\-1\\0\end{pmatrix},
    \end{equation}
    for $X_1,X_2,X_3$. Only $X_1$ and $X_2$ enter the counts ($X_3$ is removed by
    the reduction of \Cref{prop:reduction}); from \eqref{eq:blocks} and
    \eqref{eq:Rdef},
    \begin{equation}
        \T_1=\tfrac12\big(\t_{11}+2\RE\t_{12}+\t_{22}\big)\in\R,\qquad
        \T_{21}=\tfrac1{\sqrt2}\big(\overline{\t_{13}}+\overline{\t_{23}}\big),\qquad
        L_1=\tfrac{r_1}{2}.
    \end{equation}
    \Cref{thm:main} then gives
    \begin{equation}\label{eq:exn1n2}
        n_1=\begin{cases}1,&\T_1<0,\\ 0,&\T_1\ge0,\end{cases}
        \qquad
        n_2=\begin{cases}1,&\T_1=0\ \text{and}\ \T_{21}\neq0,\\
        0,&\text{otherwise},\end{cases}
    \end{equation}
    since $\ker\T_1\neq\{0\}$ only when $\T_1=0$, in which case
    $\Xi=|\T_{21}|^2$ and $L_2=\tfrac{r_1}{2}$. The corresponding leading
    coefficients are
    \begin{equation}\label{eq:excoeff}
        \l\sim-\frac{2|\T_1|}{r_1}\,\e^{-1}\quad(\T_1<0),
        \qquad
        \l\sim-\Big(\frac{2|\T_{21}|^2}{r_1}\Big)^{2/3}\e^{-2/3}
        \quad(\T_1=0,\ \T_{21}\neq0),
    \end{equation}
    read off from $-L_1^{-1/2}\T_1L_1^{-1/2}=-2\T_1/r_1$ and
    $L_2^{-1/2}\Xi L_2^{-1/2}=2|\T_{21}|^2/r_1$. \Cref{tab:neg_eig} summarizes
    the four regimes.

    This is the minimal example in which each ingredient of $n_1,n_2$ can be
    switched on independently; the off-diagonal entries $\t_{13},\t_{23}$
    (absent when the short edge is a Neumann-terminated stub rather than a loop)
    are exactly what powers the fractional rate. Setting them to zero recovers
    the two-edge case analyzed in \cite{BerBorKin_aamp23}.
\end{ex}

\begin{table}[h]
  \centering
  \setlength{\tabcolsep}{8pt}
  \begin{tabular}{|c|c|c|}
    \hline
    Conditions & $\l\sim-a\,\e^{-1}$ & $\l\sim-a\,\e^{-2/3}$\\
    \hline
    $\T_1<0$ & \checkmark\ \ $a=2|\T_1|/r_1$ & \\
    \hline
    $\T_1=0,\ \T_{21}\neq0$ & & \checkmark\ \ $a=(2|\T_{21}|^2/r_1)^{2/3}$\\
    \hline
    $\T_1=0,\ \T_{21}=0$ & & \\
    \hline
    $\T_1>0$ & & \\
    \hline
  \end{tabular}
  \caption{The vertex conditions of \Cref{ex:simplest} producing an escaping
  negative eigenvalue, summarizing \eqref{eq:exn1n2} and \eqref{eq:excoeff}. The
  last two rows produce no escaping eigenvalue.}
  \label{tab:neg_eig}
\end{table}

To see the fractional rate concretely, we specialize the vertex data and write
the conditions out.

\begin{ex}[A fractional-rate special case]\label{ex:fractional}
    Within \Cref{ex:simplest}, take the vertex operator
    \begin{equation}\label{eq:exfracTheta}
        \T=\begin{pmatrix}0&0&1\\ 0&0&1\\ 1&1&0\end{pmatrix},
    \end{equation}
    so that $\T_1=0$ and $\T_{21}=\sqrt2\neq0$: this is the second row of
    \Cref{tab:neg_eig}, giving no eigenvalue of order $\e^{-1}$ and exactly one
    of order $\e^{-2/3}$. Writing $u$ for the function on the short loop (of
    length $\ell_1=\e r_1$) and $w$ for the function on the lead, the conditions
    $QF'=\T QF$ with $Q=I_3$ read, explicitly,
    \begin{equation}\label{eq:exfracVC}
        u'(0)=w(0),\qquad u'(\ell_1)=-\,w(0),\qquad w'(0)=u(0)+u(\ell_1),
    \end{equation}
    a ``hyperbolic'' coupling of the two loop ends to the lead (derivatives taken
    into the edges).

    We derive the escaping eigenvalue directly, rather than by appeal to
    \Cref{thm:main}. For $\l=-\k^2<0$, solving $-\psi''=\l\psi$ on each edge, the
    decaying solution on the lead is $w(x)=A\,e^{-\k x}$, so $w(0)=A$ and
    $w'(0)=-\k A$. It can be shown that a negative eigenvalue must have an
    eigenfunction symmetric about the midpoint of the loop,\footnote{The
    antisymmetric modes carry Neumann conditions at both ends of the loop and
    vanish on the lead; being oscillatory, they occur only at positive energies
    embedded in the continuous spectrum, not at $\l<0$.} so
    \begin{equation}\label{eq:exusym}
        u(x)=D\cosh\!\big(\k(x-\tfrac{\ell_1}{2})\big),
    \end{equation}
    whence $u(0)=u(\ell_1)=D\cosh(\k\ell_1/2)$ and $u'(0)=-\k D\sinh(\k\ell_1/2)$.
    The first two conditions in \eqref{eq:exfracVC} then coincide and fix
    $w(0)=u'(0)=-\k D\sinh(\k\ell_1/2)$, while the third,
    $w'(0)=u(0)+u(\ell_1)$, that is $-\k\,w(0)=2D\cosh(\k\ell_1/2)$, on cancelling
    $D\neq0$ gives the secular equation
    \begin{equation}\label{eq:exsecular}
        \k^2\tanh\!\Big(\frac{\k\ell_1}{2}\Big)=2 .
    \end{equation}
    Along an escaping branch $\k\to\infty$ while $\k\ell_1=\k\e r_1\to0$ (by the a
    priori bound \eqref{eq:apriori}), so
    $\tanh(\k\ell_1/2)=\tfrac{\k\e r_1}{2}\big(1+o(1)\big)$ and
    \eqref{eq:exsecular} reads $\tfrac12\k^3\e r_1=2\big(1+o(1)\big)$. Hence
    \begin{equation}\label{eq:exfraclam}
        \k=\Big(\frac{4}{r_1}\Big)^{1/3}\e^{-1/3}\big(1+o(1)\big),
        \qquad
        \l=-\k^2=-\Big(\frac{4}{r_1}\Big)^{2/3}\e^{-2/3}\big(1+o(1)\big),
    \end{equation}
    a single negative eigenvalue escaping at the fractional rate, in agreement
    with \Cref{thm:main} and of exactly the type first observed for two-edge
    graphs in \cite{BerBorKin_aamp23}.
\end{ex}

\section{Outline of the Proof}\label{sec:outline}
The argument proceeds in four steps, each occupying a subsection of
\Cref{sec:proof}; we sketch them here, deferring precise definitions to their
proper place.

The starting point (\Cref{sec:DtN}) is the Dirichlet-to-Neumann map of $\cH_\e$.
Because it is assembled edge by edge, it is explicit, and it converts the
eigenvalue problem into a finite-dimensional one: a negative eigenvalue
$\l=-\k^2$ occurs precisely when a fixed $m\times m$ Hermitian matrix
$A(s,t)=st\,\T_V+V^*L(s)V$ becomes singular, where $s=\k\e$ and $t=\k^{-1}$
(\Cref{crl:Lcriterion}). Chasing eigenvalues of $\cH_\e$ to $-\infty$ becomes
chasing the positive solution branches $s=s(t)$ of $\det A(s,t)=0$ as
$t\to0^+$.

Since $A$ is analytic, the Newton--Puiseux theorem already forces every branch
to leave the origin at a power-law rate $s\sim c\,t^\alpha$, and the a priori
bound \eqref{eq:apriori} forces $\alpha\ge1$. The heart of the proof is to
sharpen this to \emph{two} exponents. First (\Cref{subsec:reduction}) two Schur
complements, organized by the order decomposition
$\cL_1\oplus\cL_2\oplus\cL_3$, reduce $\det A=0$ to the equivalent equation
$\det D(s,t)=0$ for a small Hermitian matrix $D(s,t)$ on $X_1$, with a positive
linear-in-$s$ leading part.

The reduced family $D$ is then fed to an \emph{implicit Rellich-type theorem}
(\Cref{sec:rellich}): combining Dencker's Hermitian preparation theorem with
Rellich's analyticity of Hermitian eigenvalues, it shows that the branches
$s=s_k(t)$ are analytic, so $\alpha\in\mathbb N$, and that their signed rates are
exactly the leading orders of the eigenvalues of the single-variable Hermitian
slice $D(0,t)$. Inspecting that slice (\Cref{subsec:tworates-proof}) shows its
negative eigenvalues have order either $t$ (outside $\ker\T_1$) or $t^2$ (on
$\ker\T_1$, outside $\ker\T_{21}$), with nothing in between. This is where the two
rates, and the counts $n_1,n_2$, come from.

Finally (\Cref{subsec:constants}) substituting the ans\"atze $s\sim ct$ and
$s\sim ct^2$ back into $\det D=0$ and separating scales by one more Schur
complement produces the leading coefficients as eigenvalues of the two Hermitian
matrices in \Cref{thm:main}.

\section{The Dirichlet-to-Neumann Map}\label{sec:DtN}
The proof of the main results rests on the DtN map of
$\cH_\e$. Since the underlying material is semi-known, we postpone its
theoretical foundation (boundary triplets, the gamma field, and the Weyl
function) to \Cref{sec:DtNbackground}, and here simply record the explicit form
of the map that we need.

\subsection{Fully disjoint DtN map}
Because the DtN map is constructed starting from pure Dirichlet conditions
that decouple every edge end, it does not see the connectivity of the graph:
it is the DtN map of the \emph{fully disjoint graph}, the disjoint union of
$m_0$ finite edges of lengths $\ell_1,\dots,\ell_{m_0}$ and $m_1$
semi-infinite leads. Consequently it is assembled edge by edge. For a single
finite edge of length $\ell$ we set
\begin{equation}\label{eq:DtNloop}
    \L_\k(\ell) := \frac{\k}{\sinh(\k\ell)}\begin{pmatrix}
        -\cosh(\k\ell) & 1\\
        1 & -\cosh(\k\ell)
    \end{pmatrix},
\end{equation}
and for a single semi-infinite lead,
\begin{equation}\label{eq:DtNlead}
    \L_\k^\infty := -\k .
\end{equation}

\begin{prop}[Fully disjoint DtN map]\label{prop:disjointDtN}
    Let
    \begin{equation}\label{eq:CR}
        \k\in\C_R:=\{z\in\C:\RE z>0\}.
    \end{equation}
    The Dirichlet-to-Neumann map of the fully disjoint graph with finite
    edges of lengths $\ell_1,\dots,\ell_{m_0}$ and $m_1$ semi-infinite leads
    is the $m\times m$ matrix, $m=2m_0+m_1$,
    \begin{equation}\label{eq:disjointDtN}
        \tL\big(\k;\ell_1,\dots,\ell_{m_0}\big)
        = \left(\bigoplus_{j=1}^{m_0}\L_\k(\ell_j)\right)
          \oplus \big(-\k\, I_{m_1}\big)
        = \begin{pmatrix}
            \L_\k(\ell_1) & & & 0\\
            & \ddots & & \\
            & & \L_\k(\ell_{m_0}) & \\
            0 & & & -\k\, I_{m_1}
        \end{pmatrix},
    \end{equation}
    where $\L_\k(\ell)$ and $\L_\k^\infty=-\k$ are given in
    \eqref{eq:DtNloop} and \eqref{eq:DtNlead}, and $I_{m_1}$ is the
    $m_1\times m_1$ identity.
\end{prop}

The per-edge formulas \eqref{eq:DtNloop} and \eqref{eq:DtNlead} and the assembly
\eqref{eq:disjointDtN} are derived from the boundary-triplet construction in
\Cref{sec:DtNbackground}.

The analytic behavior of the DtN map, combined with the vertex conditions of
$\cH_\e$, reduces the negative-eigenvalue problem to a finite-dimensional one.
Recall from \Cref{subsec:vertexcond} that the vertex conditions are encoded by
the orthogonal projector $Q$ and the self-adjoint operator $\T:\Ran Q\to\Ran Q$,
and that $V$ is the orthonormal frame for $\Ran Q$ of \Cref{subsec:decomp} with
$\T_V:=V^*\T V$.

\begin{thm}[Properties of the DtN map]\label{thm:DtNproperties}
    Write $\tL(\k):=\tL\big(\k;\ell_1,\dots,\ell_{m_0}\big)$ for the DtN map
    \eqref{eq:disjointDtN}, and let $\l=-\k^2$. Then
    \begin{enumerate}[(a)]
        \item $\tL(\k)$ is holomorphic on $\C_R$ and self-adjoint for
        $\l\in\R$ (equivalently, for $\k>0$); moreover, for $\l\in\R$,
        \begin{equation}\label{eq:Mmonotone}
            \frac{d}{d\l}\tL \geq 0,
        \end{equation}
        so $\tL$ is monotone as a function of $\l$.
        \item For every $\k_0\in\C_R$ the following are equivalent:
        \begin{enumerate}[(i)]
            \item $-\k_0^2$ is an eigenvalue of $\cH_\e$ of multiplicity $k$;
            \item $\k_0$ is a root of
            $\k\mapsto\det\big(\T_V-V^*\tL(\k)V\big)$ of multiplicity $k$;
            \item $\dim\ker\big(\T_V-V^*\tL(\k_0)V\big)=k$.
        \end{enumerate}
        \item For any $\l<0$,
        \begin{equation}\label{eq:counting}
            N\big(\cH_\e;(-\infty,\l)\big)=n_-\big(\T_V-V^*\tL(\k)V\big),
            \qquad \l=-\k^2,
        \end{equation}
        where $N\big(\cH_\e;\cdot\big)$ counts eigenvalues, with
        multiplicity, in the given interval and $n_-(\cdot)$ denotes the
        number of negative eigenvalues.
    \end{enumerate}
\end{thm}
The proof of \Cref{thm:DtNproperties} is given in \Cref{sec:DtNbackground}.

\subsection{Rescaling the DtN map}
\label{subsec:rescale}
We now introduce the small parameter through $\ell_j=\e r_j$
(see \eqref{eq:edgelength}) and set
\begin{equation}\label{eq:sdef}
    s:=\k\e .
\end{equation}
Each finite-edge block $\L_\k(\ell_j)=\L_\k(\e r_j)$ is diagonalized by the
symmetric and antisymmetric loop directions $u_1,u_2$ of \eqref{eq:loopeigvec},
with eigenvalues
\begin{equation}\label{eq:loopeigval}
    \L_\k(\e r_j)\,u_1=-\frac{s}{\e}\tanh\!\Big(\frac{s r_j}{2}\Big)u_1,
    \qquad
    \L_\k(\e r_j)\,u_2=-\frac{s}{\e}\coth\!\Big(\frac{s r_j}{2}\Big)u_2,
\end{equation}
while each semi-infinite lead contributes $-\k=-s/\e$. Factoring out $-1/\e$,
all $\e$-dependence collapses into $s$, and we write
\begin{equation}\label{eq:Ldef}
    \tL\big(\k;\{\e r_j\}\big)=-\frac{1}{\e}\,L(s),
    \qquad
    L(s):=-\e\, \tL\big(\k;\{\e r_j\}\big).
\end{equation}

The three families of eigen-directions (the even loop directions $u_1$,
the leads, and the odd loop directions $u_2$) organize $L(s)$ according to
their order in $s$:
\begin{equation}\label{eq:Ldecomp}
    L(s)=s^2 J_1(s)+s\,J_2+J_3(s),
\end{equation}
where $J_2=\hat I=0_{2m_0\times 2m_0}\oplus I_{m_1}$ is the (constant) orthogonal
projector onto the lead directions, and $J_1,J_3$ are holomorphic near $s=0$ with
\begin{equation}\label{eq:Jdef}
    J_1(s)=\bigg(\bigoplus_{j=1}^{m_0}\frac{\tanh(s r_j/2)}{s}\,u_1u_1^*\bigg)
    \oplus 0_{m_1},
    \qquad
    J_3(s)=\bigg(\bigoplus_{j=1}^{m_0} s\coth\!\Big(\frac{s r_j}{2}\Big)\,u_2u_2^*\bigg)
    \oplus 0_{m_1}.
\end{equation}
At $s=0$,
\begin{equation}\label{eq:J0}
    J_1(0)=\bigg(\bigoplus_{j=1}^{m_0}\frac{r_j}{2}\,u_1u_1^*\bigg)\oplus 0_{m_1}=R,
    \qquad
    J_3(0)=\bigg(\bigoplus_{j=1}^{m_0}\frac{2}{r_j}\,u_2u_2^*\bigg)\oplus 0_{m_1},
\end{equation}
so that the leading loop coefficient $J_1(0)$ is exactly the form $R$ of
\eqref{eq:Rdef}. The ranges of $J_1,J_2,J_3$ are the subspaces $\cL_1,\cL_2,\cL_3$
of \eqref{eq:cLdef} (the symmetric loop directions, the leads, and the
antisymmetric loop directions), mutually orthogonal and spanning the whole trace
space; for small $s>0$ the matrices $J_1(s)$ and $J_3(s)$ are positive definite
on $\cL_1$ and $\cL_3$.

Combining the rescaling \eqref{eq:Ldef} with \Cref{thm:DtNproperties}
recasts the negative-eigenvalue criterion and the counting formula directly
in terms of $L(s)$.

\begin{crl}\label{crl:Lcriterion}
    Let $s=\k\e$ and $t=\k^{-1}$, and define
    \begin{equation}\label{eq:Adef}
        A(s,t):=st\,\T_V+V^*L(s)V .
    \end{equation}
    Then, for $\l=-\k^2<0$, $\l$ is an eigenvalue of $\cH_\e$ of multiplicity
    $k$ if and only if $\det A(s,t)=0$, in which case $k=\dim\ker A(s,t)$; that
    is, the negative eigenvalues of $\cH_\e$ are exactly the values $\l=-\k^2$
    for which $A(s,t)$ is singular.
  \end{crl}

\begin{proof}
    By \eqref{eq:Ldef} we have $\tL(\k)=-\tfrac1\e L(s)$, so
    \begin{equation}
        \T_V-V^*\tL(\k)V=\T_V+\tfrac1\e V^*L(s)V .
    \end{equation}
    Since $\e=st>0$, multiplying by $\e$ scales this Hermitian matrix by a
    positive constant, so
    \begin{equation}
        A(s,t)=st\,\T_V+V^*L(s)V=\e\big(\T_V-V^*\tL(\k)V\big)
    \end{equation}
    has the same kernel as $\T_V-V^*\tL(\k)V$, and $\det A(s,t)=0$ if and only if
    $\det\big(\T_V-V^*\tL(\k)V\big)=0$. The claim now follows from
    \Cref{thm:DtNproperties}(b).
\end{proof}

\section{An implicit Rellich-type theorem}\label{sec:rellich}
The second of our two tools is a stability result for the eigenvalue
branches of a Hermitian analytic family. We state it in general form,
as it may be useful elsewhere; it is applied twice to the reduced
matrix $D(s,t)$ in \Cref{sec:proof}.

\begin{thm}[Implicit Rellich theorem]
  \label{thm:perturbedRellich}
  Let $F(\lambda,t)$ be a real-analytic $d\times d$ matrix-valued function,
  Hermitian for real $(\lambda,t)$, with
  \begin{equation}\label{eq:PRhyp}
    F(0,0)=0, \qquad \partial_\lambda F(0,0)>0 .
  \end{equation}
  Then there exist real-analytic functions $\lambda_1(t),\dots,\lambda_d(t)$
  with $\lambda_k(0)=0$ such that, in a neighborhood of the origin,
  \begin{equation}\label{eq:PRbranches}
    \det F(\lambda,t)=0
    \quad\Longleftrightarrow\quad
    \lambda\in\{\lambda_1(t),\dots,\lambda_d(t)\}.
  \end{equation}
  Moreover, with the derivatives $\partial_\lambda F$ and $\partial_t F$
  evaluated at the origin:
  \begin{enumerate}[(a)]
  \item for all sufficiently small $t>0$, the number of positive branches
    equals the negative inertia of $F(0,t)$,
    \begin{equation}\label{eq:PRcount}
      \#\{k:\lambda_k(t)>0\}=n_-\big(F(0,t)\big);
    \end{equation}
  \item the rates $\{p_k\}\subset \mathbb{N}$ occurring among the branches
    $\lambda_k\sim c_k\,t^{p_k}$ are, with multiplicity, the rates occurring
    among the eigenvalues of $F(0,t)$;
  \item \label{item:slopes} the slopes $\lambda_1'(0),\dots,\lambda_d'(0)$ are,
    with multiplicity, the eigenvalues of
    \begin{equation}
      \label{eq:HFimplicit}
      -(\partial_\lambda F)^{-1/2}\,\partial_t F\,(\partial_\lambda F)^{-1/2},
    \end{equation}
    or, equivalently, the roots $\rho$ of
    $\det(\rho\,\partial_\lambda F+\partial_t F)=0$.
  \end{enumerate}
\end{thm}

\begin{rmk}[Why ``Implicit Rellich'']
  \label{rmk:whyrellich}
  When
  $F(\lambda,t)=\lambda I-M(t)$, i.e. linear in $\lambda$ with
  $\partial_\lambda F$ equal to the identity, the equivalence
  \eqref{eq:PRbranches} is nothing but Rellich's theorem: the branches are
  the eigenvalues of the Hermitian
  one-parameter family $M(t)$, analytic by
  \cite[Thm.~II.6.1]{Kato_perturbation}, \cite[Sec.~3.5.1, Thm.~1]{Baumgartel_PertTheory}. The
  general hypothesis \eqref{eq:PRhyp} (an arbitrary positive leading coefficient
  together with higher-order terms in $\lambda$) is a perturbation of this normal
  form; the proof below strips that perturbation away, through Dencker's Hermitian
  preparation theorem, and returns the problem to the Rellich situation.

  Furthermore, the condition $F(0,0)=0$ translates into $M(0)=0$, and
  the situation is that of ``degenerate perturbation theory'',
  with the multiple eigenvalue $0$ splitting into $d$ branches whose
  slopes are the eigenvalues of $M'(0)$.  This result is known as
  the Hellmann--Feynman formula to physicists and chemists, and as the Kato
  selection theorem \cite[Thm.~II.5.4]{Kato_perturbation} to
  mathematicians.  Part \ref{item:slopes} of
  \Cref{thm:perturbedRellich} generalizes Hellmann--Feynman--Kato to
  our implicit situation, and the answer takes the form of the
  classical implicit function theorem: minus the $t$-derivative divided
  by the $\lambda$-derivative.
\end{rmk}

\begin{proof}
    By Dencker's Hermitian preparation theorem \cite[Thm.~1.1, case
    $m=1$]{Den26} (there ``symmetric'' means self-adjoint,
    $A^*=\overline A^{\,t}$), the hypotheses \eqref{eq:PRhyp} yield, near the
    origin, a real-analytic $U(\lambda,t)$ with $U(0,0)$ invertible and a
    real-analytic Hermitian $M(t)$ with $M(0)=0$ such that
    \begin{equation}\label{eq:PRfact}
        F(\lambda,t)=U(\lambda,t)\big(\lambda I-M(t)\big)U^*(\lambda,t).
    \end{equation}
    The benefit of this factorization is that we may now analyze the inner factor
    $\lambda I-M(t)$, which is \emph{linear} in $\lambda$. Taking determinants and
    using $\det U\neq0$ near the origin,
    \begin{equation}\label{eq:PRdet}
        \det F(\lambda,t)=|\det U(\lambda,t)|^2\,\det\big(\lambda I-M(t)\big),
    \end{equation}
    so the branches are the eigenvalues $\lambda_1(t),\dots,\lambda_d(t)$ of
    $M(t)$. As $M$ is Hermitian and real-analytic in the single variable $t$,
    Rellich's theorem \cite[Thm.~II.6.1]{Kato_perturbation} gives an analytic
    labeling; hence each $\lambda_k(t)$ is real-analytic with $\lambda_k(0)=0$,
    proving \eqref{eq:PRbranches}.

    Setting $\lambda=0$ in \eqref{eq:PRfact} gives
    $F(0,t)=-U(0,t)M(t)U(0,t)^*$, so $M(t)$ is congruent to
    $-F(0,t)$ through the invertible factor $U(0,t)$. Writing
    $\nu_1(t),\dots,\nu_d(t)$ for the eigenvalues of $F(0,t)$,
    Ostrowski's theorem \cite[Thm.~4.5.9]{HornJohnson} gives
    \begin{equation}\label{eq:PRostrowski}
        \lambda_k(t)=-\theta_k(t)\,\nu_k(t),\qquad
        0<c\le\theta_k(t)\le C<\infty ,
    \end{equation}
    the bounds holding as $t\to0$ since $U(0,0)$ is invertible. Thus each branch
    $\lambda_k$ and the eigenvalue $\nu_k$ have opposite sign and the same
    leading order, which is (b), and counting signs gives (a):
    \begin{equation}\label{eq:PRcountproof}
        \#\{k:\lambda_k(t)>0\}=\#\{k:\nu_k(t)<0\}=n_-\big(F(0,t)\big).
    \end{equation}

    For (c) we extract the first-order behavior from \eqref{eq:PRdet}: substitute
    $\lambda=\rho t$ and divide by $t^d$. Since
    $F(\rho t,t)=t\big(\rho\,\partial_\lambda F+\partial_t F\big)+o(t)$ (the
    derivatives at the origin) and
    $\lambda_k(t)/t\to\lambda_k'(0)$, letting $t\to0^+$ gives the polynomial
    identity
    \begin{equation}\label{eq:PRslopes}
        \det\big(\rho\,\partial_\lambda F+\partial_t F\big)
        =\det(\partial_\lambda F)\,\prod_{k=1}^d\big(\rho-\lambda_k'(0)\big),
    \end{equation}
    the leading coefficients agreeing because $U(0,0)U(0,0)^*=\partial_\lambda F$.
    As $\partial_\lambda F>0$, the roots of the left-hand side are the eigenvalues
    of $-(\partial_\lambda F)^{-1/2}\,\partial_t F\,(\partial_\lambda F)^{-1/2}$,
    so by \eqref{eq:PRslopes} these coincide, with multiplicity, with the slopes
    $\lambda_k'(0)$, which is (c).
\end{proof}

\section{Proof of the Main Result}\label{sec:proof}
Throughout this section we use the criterion of \Cref{crl:Lcriterion}: for
$\l=-\k^2<0$, the negative eigenvalues of $\cH_\e$ are exactly those $\l$ for
which
\begin{equation}\label{eq:proofcriterion}
  \det A(s,t) =
  \det \Big( st\,\T_V+V^*L(s)V\Big)=0,
  \qquad s=\k\e = \e\sqrt{-\lambda},
  \quad t=\k^{-1} = \frac1{\sqrt{-\lambda}}.
\end{equation}

\subsection{Rescaled eigenvalue branches}
Consider any continuous negative eigenvalue branch $\l(\e)=-\k(\e)^2$
of $\cH_\e$. Because $\l<0$, the quantity $\k=\sqrt{-\l}$ is real and
\emph{positive}, so both rescaled variables are positive,
\begin{equation}\label{eq:pos}
    t=\k^{-1}>0 \qquad\text{and}\qquad s=\k\e>0.
\end{equation}
The eigenvalue escapes to $-\infty$, i.e.\ $\l(\e)\to-\infty$ as $\e\to0^+$,
precisely when $\k\to+\infty$, that is, when $t\to0^+$. Moreover, the a priori
lower bound \eqref{eq:apriori}, $\l\gtrsim-\e^{-1}$ (equivalently
$\k^2\e\lesssim1$), translates into
\begin{equation}\label{eq:aprioritst}
    \frac{s}{t}=\k^2\e\lesssim 1,\qquad\text{i.e.}\qquad s\lesssim t,
\end{equation}
so along such a branch $s\to0^+$ as well.

Conversely, a positive solution $(s,t)$ of $\det A(s,t)=0$ with $s,t\to0^+$
determines, through $\k=t^{-1}$ and $\e=st$, a negative eigenvalue
$\l=-t^{-2}$ of $\cH_\e$ at parameter value $\e=st$, and this eigenvalue
escapes to $-\infty$ as $t\to0^+$. We have therefore shown that
\begin{quote}
\emph{tracking the negative eigenvalues of $\cH_\e$ that escape to $-\infty$
as $\e\to0^+$ is equivalent to tracking the positive solution branches
$s=s(t)$, with $s,t>0$, of $\det A(s,t)=0$ as $t\to0^+$.}
\end{quote}

\subsection{A priori power-law rate}
Since $L(s)$ is holomorphic near $s=0$ by \eqref{eq:Jdef}, the function
\begin{equation}\label{eq:Fdef}
    F(s,t):=\det A(s,t)
\end{equation}
is holomorphic in a neighborhood of $(0,0)$. It is not identically zero: for
each fixed $\e>0$ the negative spectrum of $\cH_\e$ is discrete, so $F$ cannot
vanish on a full neighborhood of the origin. By the Newton--Puiseux theorem
for holomorphic functions \cite[Sec.~8.3]{BK86}, the zero set of $F$ near the
origin is a finite union of branches, each of which admits a convergent
Puiseux parametrization. In particular, every solution branch approaching the
origin does so at a power-law rate: there exist
$\alpha\in\mathbb{Q}_{>0}$ and $c\neq0$ such that
\begin{equation}\label{eq:puiseux}
    s(t)=c\,t^{\alpha}\big(1+o(1)\big),\qquad t\to0^+.
\end{equation}
For the positive branches of interest we have $c>0$, and the a priori bound
\eqref{eq:aprioritst} forces the leading exponent to satisfy
\begin{equation}\label{eq:alphage1}
  \alpha\geq1.
\end{equation}

Using the particular structure of $A(s,t)$ we will establish more than the power-law form \eqref{eq:puiseux}:
each solution branch $s=s(t)$ is an \emph{analytic} function of $t$ at
$t=0$ (and, in particular, $\alpha \in \mathbb{N}$), and only two
possible exponents can occur,
\begin{equation}\label{eq:tworates}
  s\sim t\quad(\alpha=1)
  \qquad\text{and}\qquad
  s\sim t^{2}\quad(\alpha=2).
\end{equation}

\subsection{Reduction of the determinant equation}\label{subsec:reduction}
To analyze $F(s,t)=\det A(s,t)$ we use the orthogonal decomposition
$\Ran Q=X_1\oplus X_2\oplus X_3$ of \Cref{lem:decomp}, its aligned frame
$V=\begin{pmatrix}V_1&V_2&V_3\end{pmatrix}$, and the blocks $\T_1,\T_{21},L_1$ of
\eqref{eq:blocks} (with $L_1>0$ and $V_1^*\T V_2=\T_{21}^*$). We first supply the
deferred proof of the decomposition.

\begin{proof}[Proof of \Cref{lem:decomp}]
    We use the elementary identity: for subspaces $U,W$ of a Hilbert space $H$,
    \begin{equation}\label{eq:intdecomp}
        W=(W\cap U)\oplus\big(W\cap(W\cap U)^\perp\big),
    \end{equation}
    which holds since $W=W\cap\big((W\cap U)\oplus(W\cap U)^\perp\big)$ and
    $W\cap U\subseteq W$. Apply it with $W=\Ran Q\cap(\cL_1\oplus\cL_2)$ and
    $U=\cL_1$. Since $\cL_1\subseteq\cL_1\oplus\cL_2$,
    \begin{equation}\label{eq:X1inter}
        \big(\Ran Q\cap(\cL_1\oplus\cL_2)\big)\cap\cL_1=\Ran Q\cap\cL_1=X_1,
    \end{equation}
    so that \eqref{eq:intdecomp} gives
    \begin{equation}\label{eq:decompL2}
        \Ran Q\cap(\cL_1\oplus\cL_2)
        =X_1\oplus\big(\Ran Q\cap(\cL_1\oplus\cL_2)\cap X_1^\perp\big)
        =X_1\oplus X_2 .
    \end{equation}
    Since $X_1\oplus X_2\subseteq\Ran Q$, the definition
    $X_3=\Ran Q\cap(X_1\oplus X_2)^\perp$ makes
    $\Ran Q=(X_1\oplus X_2)\oplus X_3$ an orthogonal decomposition; with
    $X_2\subseteq X_1^\perp$ this is the orthogonal splitting
    \eqref{eq:RanQdecomp}.
\end{proof}

We now reduce the determinant equation \eqref{eq:proofcriterion} to a smaller
one, supported on $X_1$.
\begin{prop}\label{prop:reduction}
    There are real-analytic Hermitian matrix functions $E_2(s,t)$ and
    $E_1(s,t)$, defined near $(0,0)$ and vanishing at $(0,0)$, such
    that, introducing
    $ D(s,t): X_1\to X_1$,
    \begin{equation}\label{eq:D4}
        D(s,t)=t\,\T_{1}-t^2\,\T_{21}^*\big(I+E_2(s,t)\big)\T_{21}
        +s\big(L_1+E_1(s,t)\big),
    \end{equation}
    the equivalence
    \begin{equation}\label{eq:detequiv}
        \det A(s,t)=0
        \quad\Longleftrightarrow\quad
        \det D(s,t)=0
    \end{equation}
    holds for all sufficiently small $s,t>0$.
  \end{prop}
  
\begin{proof}
We partition $A(s,t)=st\,\T_V+V^*L(s)V$ in the aligned frame
$V=\begin{pmatrix}V_1&V_2&V_3\end{pmatrix}$ of
\Cref{lem:decomp} and reduce it by two successive Schur complements. Throughout
we use the Schur determinant identity: for a Hermitian block matrix with
invertible block $S$,
\begin{equation}\label{eq:schurdet}
    \det\begin{pmatrix}P&R\\ R^*&S\end{pmatrix}
    =\det S\cdot\det\big(P-RS^{-1}R^*\big),
\end{equation}
the second factor being the Schur complement of $S$ (see \Cref{app:schur}).

\emph{Block structure.} Insert the order decomposition
$L(s)=s^2J_1(s)+sJ_2+J_3(s)$ of \eqref{eq:Ldecomp}, whose terms are supported on
$\cL_1,\cL_2,\cL_3$. The orthogonalities $X_1\perp\cL_2\oplus\cL_3$ and
$X_2\perp\cL_3$ make the frame $V$ bring $A=st\,\T_V+V^*L(s)V$ to
\begin{equation}\label{eq:Ablocks}
    A=\begin{pmatrix}
        st\,\T_{1}+s^2\big(L_1+o(1)\big) & st\,\T_{21}^*+O(s^2) & o(s)\\[2pt]
        st\,\T_{21}+O(s^2) & s\big(I+o(1)\big) & O(s)\\[2pt]
        o(s) & O(s) & L_3+o(1)
    \end{pmatrix},
\end{equation}
with $L_1=V_1^*J_1(0)V_1>0$ and $L_3=V_3^*J_3(0)V_3>0$. Each block's order is
dictated by these orthogonalities: on $X_1$ the map $L(s)$ enters only through
$s^2J_1$; on $X_2$ through $sJ_2$; on $X_3$ through $J_3=O(1)$; and off the
diagonal, besides the vertex-condition couplings $st\,\T_{ij}$, there are the
higher-order tails of $L(s)$. Although $L(s)$ is block-diagonal on
$\cL_1\oplus\cL_2\oplus\cL_3$, the subspaces $X_i$ are not aligned with the
$\cL_j$ (for instance $X_2\subseteq\cL_1\oplus\cL_2$), so $V^*L(s)V$ acquires
off-diagonal entries, of the orders shown.

\emph{First Schur complement (against $X_3$).} For small $s,t>0$ the trailing
block $L_3+o(1)$ is invertible, so \eqref{eq:schurdet} gives $\det
A=\det(L_3+o(1))\cdot\det A_2$, where $A_2$ is the Schur complement of the
$X_3$-block:
\begin{align}\label{eq:D2}
    A_2 &=\begin{pmatrix}
        st\,\T_{1}+s^2(L_1+o(1)) & st\,\T_{21}^*+O(s^2)\\[2pt]
        st\,\T_{21}+O(s^2) & s(I+o(1))
    \end{pmatrix}
    -\binom{o(s)}{O(s)}\big(L_3+o(1)\big)^{-1}\begin{pmatrix}o(s)&O(s)\end{pmatrix}
    \notag\\
    &=s\begin{pmatrix}
        t\,\T_{1}+s\big(L_1+o(1)\big) & t\,\T_{21}^*+O(s)\\[2pt]
        t\,\T_{21}+O(s) & I+o(1)
    \end{pmatrix}=:s\,\widetilde A_2 .
\end{align}

\emph{Second Schur complement.} The trailing block $I+o(1)$ of $\widetilde A_2$
is Hermitian and invertible near the origin; set $I+E_2(s,t):=(I+o(1))^{-1}$,
real-analytic Hermitian with $E_2(0,0)=0$. Applying \eqref{eq:schurdet} to
$\widetilde A_2$, together with $\det A_2=s^{\dim(X_1\oplus X_2)}\det\widetilde
A_2$ and $\det A=\det(L_3+o(1))\det A_2$, gives
\begin{equation}\label{eq:detfull}
  \det A(s,t)=s^{\dim(X_1\oplus X_2)}\,
  \det\big(L_3+o(1)\big)\,\det\big(I+o(1)\big)\,\det D(s,t),
\end{equation}
where $D(s,t)$ is the Schur complement of $I+o(1)$ in $\widetilde A_2$,
\begin{align}
    D(s,t)
    &=t\,\T_{1}+s\big(L_1+o(1)\big)
      -\big(t\,\T_{21}^*+O(s)\big)(I+o(1))^{-1}\big(t\,\T_{21}+O(s)\big)\notag\\
    &=t\,\T_{1}-t^2\,\T_{21}^*\big(I+E_2(s,t)\big)\T_{21}+s\big(L_1+E_1(s,t)\big),
      \label{eq:Dproof}
\end{align}
$E_1(s,t)\to0$; the second line is \eqref{eq:D4}. The prefactors in
\eqref{eq:detfull} are nonzero for small $s,t>0$, so $\det A(s,t)=0$ if and only
if $\det D(s,t)=0$, which is \eqref{eq:detequiv}.
\end{proof}

\subsection{Analyticity and the two rates}\label{subsec:tworates-proof}
We apply \Cref{thm:perturbedRellich} to $F=D$, with $s$ in the role of the branch
variable $\lambda$. This
is licensed because $D$ is real-analytic and Hermitian near the origin (both
inherited from $A(s,t)$ through the analytic, Hermitian Schur-complement
reduction), and, by construction,
\begin{equation}\label{eq:D4analytichyp}
    D(0,0)=0, \qquad \partial_s D(0,0)=L_1>0.
\end{equation}

\begin{lem}[Analytic branches and the two rates]\label{lem:analytic}
    Every solution branch $s=s(t)$ of $\det D(s,t)=0$ approaching the origin is
    real-analytic at $t=0$ with $s(0)=0$; in particular the leading exponent
    $\alpha$ of \eqref{eq:puiseux} is a positive integer. Moreover, for all
    sufficiently small $t>0$ there are exactly
    \begin{enumerate}
    \item 
      $n_1=n_-(\T_1)$ positive branches with $s\sim ct$ ($\alpha=1$), and
    \item
      $n_2=n_+(\Xi)$ positive branches with $s\sim ct^2$ ($\alpha=2$),
    \end{enumerate}
    and no positive branch of any higher order.
\end{lem}
\begin{proof}
    By \eqref{eq:D4analytichyp} the family $D$ satisfies the hypotheses
    \eqref{eq:PRhyp} of \Cref{thm:perturbedRellich}, so its branches $s=s_k(t)$
    are among the real-analytic branches furnished there, with
    $s_k(0)=0$; through the equivalence \eqref{eq:detequiv} of
    \Cref{prop:reduction} and the criterion \eqref{eq:proofcriterion} the
    escaping branches of $\cH_\e$ are analytic in $t$, and $\alpha$ in
    \eqref{eq:puiseux} is an integer.

    For the counts we examine the unperturbed slice
    \begin{equation}\label{eq:D0t}
        D(0,t)=t\,\T_{1}-t^2\,\T_{21}^*\big(I+E_2(0,t)\big)\T_{21},\qquad
        E_2(0,0)=0,
    \end{equation}
    whose super-linear part is sandwiched between $\T_{21}$ and $\T_{21}^*$ (see
    \eqref{eq:Dproof}). For small $t>0$ its eigenvalues fall into three groups:
    \begin{itemize}
        \item on $(\ker\T_{1})^\perp$ the term $t\,\T_{1}$ dominates: order $t$,
        contributing $n_-(\T_{1})=n_1$ negative eigenvalues;
        \item on $\ker\T_{1} \cap (\ker\T_{21})^\perp$ the linear term vanishes and
        $-t^2\,\T_{21}^*(I+E_2(0,t))\T_{21}\preceq0$ dominates: order $t^2$, contributing
        $\dim\ker\T_{1}-\dim(\ker\T_{1}\cap\ker\T_{21})=n_2$ negative eigenvalues;
      \item on $\ker\T_{1}\cap\ker\T_{21}$ the entire matrix $D(0,t)$
        vanishes identically, so the eigenvalues are identically $0$.
    \end{itemize}
    Hence $n_-\big(D(0,t)\big)=n_1+n_2$ for small $t>0$. By
    \Cref{thm:perturbedRellich}(a) there are exactly $n_1+n_2$ positive branches,
    and by (b) their orders are those of the negative eigenvalues of $D(0,t)$:
    $n_1$ of order $t$ and $n_2$ of order $t^2$, with none higher (the remaining
    eigenvalues of $D(0,t)$ vanish identically).
  \end{proof}

\begin{rmk}[What forces the two rates]\label{rmk:loadbearing}
It is tempting to reason ``$D(0,t)$ is quadratic in $t$, hence two rates,'' but
neither half holds up. $D(0,t)$ is not quadratic: by \eqref{eq:D0t} it
will have $t^3$ and higher terms as soon as $\T_{22}:=V_2^*\T
V_2\neq0$.  Furthermore, a genuine quadratic need not confine the rates either: the
indefinite $\left(\begin{smallmatrix} t & t^2\\ t^2 & 0\end{smallmatrix}\right)$
already has an eigenvalue of order $-t^3$.  What the proof of
\Cref{lem:analytic} uses instead is the \emph{structure} of the super-linear
coefficient: on $\ker\T_1$ it is $-t^2\,\T_{21}^*(I+E_2)\T_{21}$,
positive-semidefinite with kernel $\ker\T_{21}$ and rank constant in $t$ (since
$I+E_2>0$). It therefore vanishes to all orders on $\ker\T_{21}$ and is bounded
below outside of $\ker\T_{21}$, so no eigenvalue can drift to $0$ to open a $t^3$ rate. The
counterexample above produces a $t^3$ rate precisely because its indefinite
coefficient drops rank.
\end{rmk}

\subsection{Leading coefficients of the escaping branches}\label{subsec:constants}
Part (a) of the implicit Rellich theorem fixed the two \emph{rates}; its part
(c) also delivers the constants multiplying them: directly for the rate $t$, and
after one Schur reduction back to the same normal form for the rate $t^2$.

\begin{lem}[Leading coefficients]\label{lem:constants}
    With $W_2$ an orthonormal frame for $\ker\T_{1}$ and $L_2,\Xi$ the
    compressions of $L_1$ and $\T_{21}^*\T_{21}$ to $\ker\T_1$ of
    \eqref{eq:kerblocks}, the positive branches of $\det D(s,t)=0$ (equivalently,
    the negative eigenvalues of $\cH_\e$ escaping to $-\infty$) are, to leading
    order:
    \begin{enumerate}[(i)]
        \item \emph{(Rate $t$.)} The $n_1=n_-(\T_{1})$ branches $s=ct+o(t)$ have
        coefficients $c>0$ equal to the positive eigenvalues of the Hermitian
        matrix $-L_1^{-1/2}\T_{1}L_1^{-1/2}$; correspondingly
        \begin{equation}\label{eq:lamS}
            \l(\e)=-c\,\e^{-1}\big(1+o(1)\big).
        \end{equation}
        \item \emph{(Rate $t^2$.)} The $n_2=n_+(\Xi)$ branches $s=ct^2+o(t^2)$
        have coefficients $c>0$ equal to the positive eigenvalues of the
        Hermitian matrix $L_2^{-1/2}\,\Xi\,L_2^{-1/2}$; correspondingly
        \begin{equation}\label{eq:lamC}
            \l(\e)=-c^{2/3}\,\e^{-2/3}\big(1+o(1)\big).
        \end{equation}
    \end{enumerate}
    The coefficient lists are independent of the choice of frame $W_2$.
\end{lem}

\begin{proof}
    By \Cref{prop:reduction} the escaping branches solve $\det D(s,t)=0$, and by
    \Cref{lem:analytic} each is analytic with $s\sim ct$ or $s\sim ct^2$, $c>0$.
    Recall from \eqref{eq:D4} that
    \begin{equation}\label{eq:Drecall}
        D(s,t)=t\,\T_{1}-t^2\,\T_{21}^*\big(I+E_2(s,t)\big)\T_{21}
        +s\big(L_1+E_1(s,t)\big),\qquad E_1,E_2\to0 .
    \end{equation}
    Both coefficients are read off from \Cref{thm:perturbedRellich}(c), applied
    first to $D$ and then to a reduction of it.
    
    \emph{(i) Rate $t$.} Apply \Cref{thm:perturbedRellich} to $F=D$, with $s$ in
    the role of the branch variable $\lambda$; by \eqref{eq:Drecall} its
    derivatives at the origin are $\partial_s D(0,0)=L_1>0$ and
    $\partial_t D(0,0)=\T_1$. By part (c) the slopes $s_k'(0)$ are, with
    multiplicity, the eigenvalues of $-L_1^{-1/2}\T_1L_1^{-1/2}$. The linear
    branches $s\sim ct$ are those with $c=s_k'(0)\neq0$.
    There are (using Sylvester's law of inertia, \Cref{app:schur})
    \begin{equation}
      \label{eq:n1_calc}
      n_1 := n_+\left(-L_1^{-1/2}\T_1L_1^{-1/2}\right)
      = n_-(\T_1)
    \end{equation}
    positive linear branches.
    Finally $s=ct$ gives $\k^2\e=c$, so $\l=-\k^2=-c\,\e^{-1}$, which is
    \eqref{eq:lamS}.

    \emph{(ii) Rate $t^2$.} The rate-$t^2$ branches are those with zero slope,
    $s_k'(0)=0$. To resolve them, substitute $s=\rho t$ into \eqref{eq:Drecall}:
    \begin{equation}\label{eq:Grhot}
        D(\rho t,t)=t\,G(\rho,t),\qquad
        G(\rho,t)=\T_1+\rho L_1-t\,\T_{21}^*(I+E_2)\T_{21}+\rho\,E_1 ,
    \end{equation}
    and seek the solutions $\rho(t)=O(t)$ of $\det G(\rho,t)=0$ (those
    with $s=\rho t=O(t^2)$). We cannot apply
    \Cref{thm:perturbedRellich} to $G$ directly, since $G(0,0)=\T_1$
    need not vanish.  We first pass to another Schur complement.  In
    the frame $W=\begin{pmatrix}W_1&W_2\end{pmatrix}$, where $W_1$
    spans $\Ran\T_1$ and $W_2$ spans $\ker\T_1$, and using
    $\T_1W_2=0$,
    \begin{equation}\label{eq:Gblock}
        W^*G(\rho,t)W=\begin{pmatrix}
            W_1^*\T_1W_1+O(\rho)+O(t) & O(\rho)+O(t)\\[2pt]
            O(\rho)+O(t) & \rho L_2-t\,\Xi+o(\rho)+o(t)
        \end{pmatrix},
    \end{equation}
    with $L_2=W_2^*L_1W_2$ and $\Xi=W_2^*\T_{21}^*\T_{21}W_2$ as in
    \eqref{eq:kerblocks}. The $(1,1)$ block is invertible near the origin, and the
    Schur complement of this leading block (the $(1,1)$-block in
    \Cref{app:schur}) is
    \begin{equation}\label{eq:Hred}
        H(\rho,t)=\rho L_2-t\,\Xi+o(\rho)+o(t)
    \end{equation}
    (the off-diagonal correction being $O((\rho+t)^2)=o(\rho)+o(t)$). Now $H$ is
    real-analytic and Hermitian with $H(0,0)=0$ and $\partial_\rho H(0,0)=L_2>0$,
    so \Cref{thm:perturbedRellich} applies, with $\rho$ in the role of $\lambda$
    and derivatives $\partial_\rho H=L_2$, $\partial_t H=-\Xi$ at the
    origin. By part (c) the branches $\rho=\rho_k(t)$ have slopes $\rho_k'(0)$
    equal, with multiplicity, to the eigenvalues of
    \begin{equation}\label{eq:rho2slopes}
        -L_2^{-1/2}(-\Xi)L_2^{-1/2}=L_2^{-1/2}\,\Xi\,L_2^{-1/2}\ \ge0 .
    \end{equation}
    Hence $s=\rho t\sim\rho_k'(0)\,t^2$.  The positive branches number
    $n_+(\Xi)=n_2$, while $\ker\Xi=\ker\T_1\cap\ker\T_{21}$ gives the
    identically-zero branches. Finally $s=ct^2$ gives $\k^3\e=c$, so
    $\l=-\k^2=-c^{2/3}\e^{-2/3}$, which is \eqref{eq:lamC}.
\end{proof}

\subsection{Proof of the main result}\label{subsec:mainproof}
We now assemble the pieces. \Cref{crl:Lcriterion} identifies the negative eigenvalues
of $\cH_\e$ escaping to $-\infty$ with the positive branches $s=s(t)\to0^+$ of
$\det A(s,t)=0$, equivalently, by \Cref{prop:reduction}, of $\det D(s,t)=0$. By
\Cref{subsec:tworates-proof} these branches are analytic in $t$ and occur at
exactly the two rates $s\sim t$ and $s\sim t^2$, with $n_1=n_-(\T_1)$ of the
first kind and $n_2=n_+(\Xi)$ of the second; through $\l=-t^{-2}$ and $\e=st$
these are the eigenvalue rates $\l\sim-\e^{-1}$ and $\l\sim-\e^{-2/3}$. This
proves the two families of \Cref{thm:main} and the counts \eqref{eq:n2}, the
identity $n_+(\Xi)=\dim\ker\T_1-\dim(\ker\T_1\cap\ker\T_{21})$ holding because
$\ker\Xi=\ker\T_1\cap\ker\T_{21}$. The leading coefficients of
\Cref{thm:main} are those computed in \Cref{lem:constants}: for a branch
$s\sim ct$ one has $a_i=c$ by \eqref{eq:lamS}, and for $s\sim ct^2$ one has
$a_i=c^{2/3}$ by \eqref{eq:lamC}.

It remains to record the analytic structure of each branch in $\e$. By
\Cref{lem:analytic} an escaping branch is analytic, $s=s(t)$ with $s(0)=0$, and
either $s(t)=ct\big(1+O(t)\big)$ (first family) or $s(t)=ct^2\big(1+O(t)\big)$
(second family), $c>0$. Since $\e=t\,s(t)$ and $\l=-t^{-2}$, for the first family
$\e(t)=ct^2\big(1+O(t)\big)$ has a double zero at $t=0$, so
$\e^{1/2}=t\sqrt{c}\,\big(1+O(t)\big)$ is analytic in $t$ with nonvanishing
derivative there; inverting, $t$ is analytic in $\e^{1/2}$, and hence
$\l=-t^{-2}$ is meromorphic in $\e^{1/2}$ with a double pole,
\begin{equation}\label{eq:merom1}
    \l(\e)=-c\,\e^{-1}\big(1+O(\e^{1/2})\big).
\end{equation}
For the second family $\e(t)=ct^3\big(1+O(t)\big)$ has a triple zero, so $t$ is
analytic in $\e^{1/3}$ and $\l=-t^{-2}$ is meromorphic in $\e^{1/3}$ with a
double pole,
\begin{equation}\label{eq:merom2}
    \l(\e)=-c^{2/3}\,\e^{-2/3}\big(1+O(\e^{1/3})\big).
\end{equation}
This is the meromorphy asserted in \Cref{thm:main}, and refines
\eqref{eq:typeS}, \eqref{eq:typeC}. \qed

\appendix
\crefalias{section}{appendix}
\section{The Dirichlet-to-Neumann map: background and proofs}\label{sec:DtNbackground}
Here we prove \Cref{prop:disjointDtN} and \Cref{thm:DtNproperties} using the
boundary-triplet formalism, following \cite[Ch.~14]{Schmudgen_unboundedSAO}.

\subsection{Boundary triplets and the Weyl function}
\begin{defn}\label{def:boundarytriple}
    Let $T$ be a densely defined symmetric operator on a Hilbert space $\cH$. A
    \emph{boundary triplet} for $T^*$ is a triple $(\cK,\G_0,\G_1)$, with $\cK$
    a Hilbert space and $\G_0,\G_1:\cD(T^*)\to\cK$ linear, such that
    \begin{enumerate}[(a)]
        \item $\langle T^*x,y\rangle-\langle x,T^*y\rangle
        =\langle\G_1x,\G_0y\rangle-\langle\G_0x,\G_1y\rangle$ for all
        $x,y\in\cD(T^*)$, and
        \item $x\mapsto(\G_0x,\G_1x)$ maps $\cD(T^*)$ onto $\cK\oplus\cK$.
    \end{enumerate}
\end{defn}

Take $T=-\Delta$ with the minimal domain $\cD(T)=H^2_0(G_\e)$, so that
$\cD(T^*)=H^2(G_\e)$, and let $\G_0f=F$, $\G_1f=F'$ be the Dirichlet and Neumann
traces of \eqref{eq:DNtrace}. Green's identity (integration by parts on each
edge, using the inward sign convention in $F'$) is exactly
\Cref{def:boundarytriple}(a), and surjectivity (b) is elementary; thus
$(\C^m,\G_0,\G_1)$ is a boundary triplet for $T^*$. Let $T_0$ be the
self-adjoint extension with $\cD(T_0)=\ker\G_0$ (functions vanishing at every
edge end). For $\l\in\rho(T_0)$ the \emph{gamma field}
\begin{equation}\label{eq:gammafield}
    \g(\l)=\big(\G_0|_{\ker(T^*-\l)}\big)^{-1}
\end{equation}
maps Dirichlet data to the corresponding eigenfunction, and the \emph{Weyl
function} (Dirichlet-to-Neumann map)
\begin{equation}\label{eq:weyl}
    M_\l=\G_1\g(\l)
\end{equation}
returns its Neumann data, so that $F'=M_\l F$ for solutions of $-\Delta f=\l f$.
Both are holomorphic on $\rho(T_0)$, and \cite[Prop.~14.15]{Schmudgen_unboundedSAO}
\begin{equation}\label{eq:DtNderivative}
    M_\l^*=M_{\bar\l},\qquad \frac{d}{d\l}M_\l=\g(\bar\l)^*\g(\l).
\end{equation}

\subsection{Per-edge computation (proof of \Cref{prop:disjointDtN})}
The minimal operator $T=-\Delta$ on $\cD(T)=H^2_0(G_\e)$ is the direct sum
$\bigoplus_e T_e$ of the per-edge minimal operators, its domain imposing no
coupling between edges; hence $T^*$, the gamma field, and the Weyl function
$M_\l$ all decompose as direct sums over the edges, and
$M_\l=\bigoplus_e\L_\k^{(e)}$ is assembled edge by edge. Fix $\l=-\k^2$ with
$\k\in\C_R$, and solve $-u''=\l u$, i.e.\ $u''=\k^2u$, on each edge.

For a \emph{semi-infinite lead} $[0,\infty)$ the $L^2$ solution is
$u=a\,e^{-\k x}$, so $\g(-\k^2)a=a\,e^{-\k x}$ and the DtN value is
\begin{equation}\label{eq:appDtNlead}
    u'(0)/u(0)=-\k ,
\end{equation}
which is $\L_\k^\infty$ of \eqref{eq:DtNlead}. For a \emph{finite edge}
$[0,\ell]$ the solution with prescribed endpoint values $u(0)=a$, $u(\ell)=b$ is
\begin{equation}
    \g(-\k^2)\begin{pmatrix}a\\b\end{pmatrix}
    =\frac{\sinh(\k(\ell-x))}{\sinh(\k\ell)}\,a+\frac{\sinh(\k x)}{\sinh(\k\ell)}\,b ,
\end{equation}
and the inward derivatives $\big(u'(0),-u'(\ell)\big)$ are given, as a map of
$(a,b)$, by
\begin{equation}\label{eq:appDtNloop}
    \frac{\k}{\sinh(\k\ell)}\begin{pmatrix}
        -\cosh(\k\ell)&1\\ 1&-\cosh(\k\ell)\end{pmatrix},
\end{equation}
which is $\L_\k(\ell)$ of \eqref{eq:DtNloop}. Taking the direct sum over the
$m_0$ loops and $m_1$ leads yields \eqref{eq:disjointDtN}.

\subsection{Proof of \Cref{thm:DtNproperties}}
\emph{(a)} Since $\sigma(T_0)\subseteq[0,\infty)$, we have
$\rho(T_0)\supseteq\C\setminus[0,\infty)$, and with the principal branch of the
square root $\l=-\k^2\in\rho(T_0)$ for every $\k\in\C_R$. Hence $\g$, and so
$M_\l$, are holomorphic there, and $M_\l^*=M_{\bar\l}$ by
\eqref{eq:DtNderivative} makes $\tL(\k)=M_{-\k^2}$ self-adjoint for real $\l$.
Finally $dM_\l/d\l=\g(\bar\l)^*\g(\l)\ge0$ is \eqref{eq:Mmonotone}.

\emph{(b)} By the Krein-type resolvent formula for boundary triplets
\cite[Prop.~14.17]{Schmudgen_unboundedSAO}, for $\k_0\in\C_R$
\begin{equation}\label{eq:dimkers}
    \dim\ker(\cH_\e+\k_0^2)=\dim\ker\big(\T_V-V^*M_{-\k_0^2}V\big),
\end{equation}
which is (i)$\Leftrightarrow$(iii). For (ii)$\Leftrightarrow$(iii), set
$B(\k)=\T_V-V^*M_{-\k^2}V$; it is Hermitian for $\k>0$, so by Rellich's theorem
\cite[Thm.~II.6.1]{Kato_perturbation} its eigenvalues $\mu_j(\k)$ and eigenvectors $u_j(\k)$
are real-analytic. If $\mu_j(\k_0)=0$ then, by \eqref{eq:DtNderivative},
$B'(\k)=2\k\,V^*(dM_\l/d\l)V\ge0$ is positive definite along the relevant
directions, so
\begin{equation}
    \mu_j'(\k_0)=\langle u_j(\k_0),B'(\k_0)u_j(\k_0)\rangle>0 ;
\end{equation}
hence each vanishing branch has a simple zero at $\k_0$, and the order of
$\k_0$ as a root of $\det B(\k)=\prod_j\mu_j(\k)$ equals the number
$\dim\ker B(\k_0)$ of vanishing branches. This is (ii)$\Leftrightarrow$(iii).

\emph{(c)} The eigenvalue-counting identity
$N\big(\cH_\e;(-\infty,\l)\big)=n_-\big(\T_V-V^*M_\l V\big)$ follows from the
symplectic (Duistermaat index) count of \cite[\S5,\;Eq.~(7.13)]{BerCoxLatSuk_jst26}, applied
to the Lagrangian planes of the Dirichlet extension $T_0$, of $\cH_\e$, and of
the graph of $M_\l$; the monotonicity \eqref{eq:Mmonotone} guarantees the index
is the negative inertia of $\T_V-V^*M_\l V$. \qed

\section{Congruence and Schur complements}\label{app:schur}
We record the two linear-algebra facts used in \Cref{sec:proof}. For
Hermitian $H$, we write $n_+(H),n_-(H),n_0(H)$ for the numbers of
positive, negative, and zero eigenvalues. \emph{Sylvester's law of
  inertia} \cite[Thm.~4.5.8]{HornJohnson}: congruence by an invertible
matrix, $H\mapsto S^*HS$, preserves the triple $(n_+,n_-,n_0)$. Two
consequences are used in \Cref{sec:proof}. First, the compression
$U^*HU$ of $H$ by an orthonormal frame $U$ of a subspace has
eigenvalues, hence inertia, independent of the frame (two orthonormal
frames of a subspace differ by a unitary); this is what makes the
counts $n_1=n_-(\T_1)$ and $n_2=n_+(\Xi)$ well defined. Second, the
sign counts are unchanged under the invertible congruences appearing
there, for instance
$n_+\big(-L_1^{-1/2}\T_1L_1^{-1/2}\big)=n_+(-\T_1)=n_-(\T_1)$ and
$n_+\big(L_2^{-1/2}\Xi L_2^{-1/2}\big)=n_+(\Xi)$.

\begin{lem}[Schur determinant identity]\label{lem:schurdet}
    For a Hermitian block matrix with invertible trailing block $S$,
    \begin{equation}\label{eq:appschurdet}
        \det\begin{pmatrix}P&R\\ R^*&S\end{pmatrix}
        =\det S\cdot\det\big(P-RS^{-1}R^*\big).
    \end{equation}
\end{lem}
\begin{proof}
    The block factorization
    \begin{equation}
        \begin{pmatrix}P&R\\ R^*&S\end{pmatrix}
        =\begin{pmatrix}I&RS^{-1}\\ 0&I\end{pmatrix}
         \begin{pmatrix}P-RS^{-1}R^*&0\\ 0&S\end{pmatrix}
         \begin{pmatrix}I&0\\ S^{-1}R^*&I\end{pmatrix}
    \end{equation}
    has unit-determinant outer factors, so the determinant is that of the middle
    factor, namely $\det(P-RS^{-1}R^*)\det S$.
\end{proof}

This is \eqref{eq:schurdet}; the second factor is the Schur complement of the
trailing block $S$, used twice in the reduction of \Cref{prop:reduction}. The
complement of the \emph{leading} block has the symmetric analogue: for invertible
$P$,
\begin{equation}\label{eq:appschurdet2}
    \det\begin{pmatrix}P&R\\ R^*&S\end{pmatrix}
    =\det P\cdot\det\big(S-R^*P^{-1}R\big),
\end{equation}
and the factor $S-R^*P^{-1}R$ is the form used for the $(1,1)$-block reduction in
the rate-$t^2$ analysis of \Cref{lem:constants}.

\bibliographystyle{siam}
\bibliography{bk_bibl,additional,references}

\def\cprime{$'$} \def\cprime{$'$} \def\cprime{$'$} \def\cprime{$'$}
  \def\cprime{$'$} \def\cprime{$'$} \def\cprime{$'$}
  \def\polhk#1{\setbox0=\hbox{#1}{\ooalign{\hidewidth
  \lower1.5ex\hbox{`}\hidewidth\crcr\unhbox0}}} \def\cprime{$'$}
  \def\cprime{$'$}
\begin{thebibliography}{10}

\bibitem{Alb_anp12}
{\sc S.~Albeverio and S.~Kusuoka}, {\em Diffusion processes in thin tubes and
  their limits on graphs}, The Annals of Probability, 40 (2012).

\bibitem{BanLev_ahp17}
{\sc R.~Band and G.~L\'evy}, {\em Quantum graphs which optimize the spectral
  gap}, Ann. Henri Poincar\'e, 18 (2017), pp.~3269--3323.

\bibitem{Baumgartel_PertTheory}
{\sc H.~Baumg\"artel}, {\em Analytic perturbation theory for matrices and
  operators}, vol.~15 of Operator Theory: Advances and Applications,
  Birkh\"auser Verlag, Basel, 1985.

\bibitem{BerBorKin_aamp23}
{\sc G.~Berkolaiko, D.~I. Borisov, and M.~King}, {\em Exotic eigenvalues and
  analytic resolvent for a graph with a shrinking edge}, Anal. Math. Phys., 13
  (2023), p.~90.

\bibitem{BerCdV_jmaa24}
{\sc G.~Berkolaiko and Y.~Colin~de Verdi\`ere}, {\em Exotic eigenvalues of
  shrinking metric graphs}, J. Math. Anal. Appl., 534 (2024), pp.~Paper No.
  128040, 9.

\bibitem{BerCoxLatSuk_jst26}
{\sc G.~Berkolaiko, G.~Cox, Y.~Latushkin, and S.~Sukhtaiev}, {\em The
  {D}uistermaat index and eigenvalue interlacing for self-adjoint extensions of
  a symmetric operator}, J. Spectr. Theory, 16 (2026), pp.~1--49.

\bibitem{BerKuc_incol12}
{\sc G.~Berkolaiko and P.~Kuchment}, {\em Dependence of the spectrum of a
  quantum graph on vertex conditions and edge lengths}, in Spectral Geometry,
  vol.~84 of Proceedings of Symposia in Pure Mathematics, American Math. Soc.,
  2012.
\newblock preprint {\tt arXiv:1008.0369}.

\bibitem{BerKuc_graphs}
\leavevmode\vrule height 2pt depth -1.6pt width 23pt, {\em Introduction to
  Quantum Graphs}, vol.~186 of Mathematical Surveys and Monographs, AMS, 2013.

\bibitem{BerLatSuk_am19}
{\sc G.~Berkolaiko, Y.~Latushkin, and S.~Sukhtaiev}, {\em Limits of quantum
  graph operators with shrinking edges}, Adv. Math., 352 (2019), pp.~632--669.

\bibitem{BerLiu_jmaa17}
{\sc G.~Berkolaiko and W.~Liu}, {\em Simplicity of eigenvalues and
  non-vanishing of eigenfunctions of a quantum graph}, J. Math. Anal. Appl.,
  445 (2017), pp.~803--818.
\newblock preprint {\tt arXiv:1601.06225}.

\bibitem{BolEnd_ahp09}
{\sc J.~Bolte and S.~Endres}, {\em The trace formula for quantum graphs with
  general self adjoint boundary conditions}, Ann. Henri Poincar\'e, 10 (2009),
  pp.~189--223.

\bibitem{Bor_math21}
{\sc D.~I. Borisov}, {\em Spectra of elliptic operators on quantum graphs with
  small edges}, Mathematics, 9 (2021), p.~1874.

\bibitem{Bor_am22}
{\sc D.~I. Borisov}, {\em Analyticity of resolvents of elliptic operators on
  quantum graphs with small edges}, Adv. Math., 397 (2022), pp.~Paper No.
  108125, 48.

\bibitem{BK86}
{\sc E.~Brieskorn and H.~Knörrer}, {\em Plane Algebraic Curves}, Springer
  Basel, 1986.

\bibitem{Cac_s19}
{\sc C.~Cacciapuoti}, {\em Scale invariant effective {H}amiltonian for a graph
  with a small compact core}, Symmetry, 11 (2019), pp.~359, 29.

\bibitem{Car_nhm11}
{\sc R.~Carlson}, {\em Spectral theory for nonconservative transmission line
  networks}, Netw. Heterog. Media, 6 (2011), pp.~257--277.

\bibitem{CheExnTur_anp10}
{\sc T.~Cheon, P.~Exner, and O.~Turek}, {\em Approximation of a general
  singular vertex coupling in quantum graphs}, Ann. Physics, 325 (2010),
  pp.~548--578.

\bibitem{CdV_ahp15}
{\sc Y.~Colin~de Verdi\`{e}re}, {\em Semi-classical measures on quantum graphs
  and the {G}au\ss{} map of the determinant manifold}, Annales Henri
  Poincar\'{e}, 16 (2015), pp.~347--364.
\newblock also {\tt arXiv:1311.5449}.

\bibitem{Den26}
{\sc N.~Dencker}, {\em Symmetric preparation of systems}, 2026.
\newblock preprint {\tt arXiv:2601.12458}.

\bibitem{DoKucOng_ems17}
{\sc N.~T. Do, P.~Kuchment, and B.~Ong}, {\em On resonant spectral gaps in
  quantum graphs}, in Functional analysis and operator theory for quantum
  physics, EMS Ser. Congr. Rep., Eur. Math. Soc., Z\"{u}rich, 2017,
  pp.~213--222.

\bibitem{HornJohnson}
{\sc R.~A. Horn and C.~R. Johnson}, {\em Matrix analysis}, Cambridge University
  Press, Cambridge, second~ed., 2013.

\bibitem{Kato_perturbation}
{\sc T.~Kato}, {\em Perturbation theory for linear operators}, Classics in
  Mathematics, Springer-Verlag, Berlin, 1995.
\newblock Reprint of the 1980 edition.

\bibitem{KosSch_jpa99}
{\sc V.~Kostrykin and R.~Schrader}, {\em Kirchhoff's rule for quantum wires},
  J. Phys. A, 32 (1999), pp.~595--630.

\bibitem{KosSch_incol06}
{\sc V.~Kostrykin and R.~Schrader}, {\em Laplacians on metric graphs:
  eigenvalues, resolvents and semigroups}, in Quantum graphs and their
  applications, vol.~415 of Contemp. Math., Amer. Math. Soc., Providence, RI,
  2006, pp.~201--225.

\bibitem{Kuc_incol01}
{\sc P.~Kuchment}, {\em The mathematics of photonic crystals}, in Mathematical
  modeling in optical science, G.~Bao, L.~Cowsar, and W.~Masters, eds., vol.~22
  of Frontiers Appl. Math., SIAM, Philadelphia, PA, 2001, pp.~207--272.

\bibitem{Kuc_wrm04}
\leavevmode\vrule height 2pt depth -1.6pt width 23pt, {\em Quantum graphs. {I}.
  {S}ome basic structures}, Waves Random Media, 14 (2004), pp.~S107--S128.
\newblock Special section on quantum graphs.

\bibitem{KucZha_jmp19}
{\sc P.~Kuchment and J.~Zhao}, {\em Analyticity of the spectrum and
  {D}irichlet-to-{N}eumann operator technique for quantum graphs}, J. Math.
  Phys., 60 (2019), pp.~093502, 8.

\bibitem{Kur_jmp13}
{\sc P.~Kurasov}, {\em Inverse scattering for lasso graph}, J. Math. Phys., 54
  (2013), pp.~042103, 14.

\bibitem{LawTanChr_sr22}
{\sc T.~Lawrie, G.~Tanner, and D.~Chronopoulos}, {\em A quantum graph approach
  to metamaterial design}, Scientific Reports, 12 (2022), p.~18006.

\bibitem{LR13}
{\sc P.~H. Lin and J.~Ro}, {\em Vibration analysis of planar serial-frame
  structures}, Journal of Sound and Vibration, 262 (2003), pp.~1113--1131.

\bibitem{Mei19}
{\sc C.~Mei}, {\em Analysis of in- and out-of plane vibrations in a rectangular
  frame based on two- and three-dimensional structural models}, Journal of
  Sound and Vibration, 440 (2019), pp.~412--438.

\bibitem{Mugnolo_book}
{\sc D.~Mugnolo}, {\em Semigroup methods for evolution equations on networks},
  Understanding Complex Systems, Springer, Cham, 2014.

\bibitem{Nic_incol85}
{\sc S.~Nicaise}, {\em Some results on spectral theory over networks, applied
  to nerve impulse transmission}, in Orthogonal polynomials and applications
  ({B}ar-le-{D}uc, 1984), vol.~1171 of Lecture Notes in Math., Springer,
  Berlin, 1985, pp.~532--541.

\bibitem{Ong_diss}
{\sc B.-S. Ong}, {\em Spectral Problems of Optical Waveguides and Quantum
  Graphs}, PhD thesis, Texas A\&M University, 2006.

\bibitem{SarCarAnd_jmb14}
{\sc J.~Sarhad, R.~Carlson, and K.~E. Anderson}, {\em Population persistence in
  river networks}, J. Math. Biol., 69 (2014), pp.~401--448.

\bibitem{SchKot_prl03}
{\sc H.~Schanz and T.~Kottos}, {\em Scars on quantum networks ignore the
  {L}yapunov exponent}, Phys. Rev. Lett., 90 (2003), p.~234101.

\bibitem{Schmudgen_unboundedSAO}
{\sc K.~Schm\"{u}dgen}, {\em Unbounded self-adjoint operators on {H}ilbert
  space}, vol.~265 of Graduate Texts in Mathematics, Springer, Dordrecht, 2012.

\bibitem{Sim_jfa77}
{\sc B.~Simon}, {\em On the absorption of eigenvalues by continuous spectrum in
  regular perturbation problems.}, J. Functional Analysis, 25 (1977),
  pp.~338--344.

\bibitem{SmiSol_conm06}
{\sc U.~Smilansky and M.~Solomyak}, {\em The quantum graph as a limit of a
  network of physical wires}, in Quantum graphs and their applications,
  vol.~415 of Contemp. Math., Amer. Math. Soc., Providence, RI, 2006,
  pp.~283--291.

\bibitem{WilWit_ijms70}
{\sc F.~W. Williams and W.~H. Wittrick}, {\em An automatic computational
  procedure for calculating natural frequencies of skeletal structures}, Int.
  J. Mech. Sci., 12 (1970), pp.~781--791.

\end{thebibliography}

\end{document}